\documentclass[12pt]{article}
\usepackage[utf8]{inputenc}

\usepackage{amsmath,amsfonts,amssymb,amsthm,bm,color,mathrsfs}
\usepackage{cleveref}

\usepackage[numbers,sort&compress]{natbib}
\newcommand{\A}{\boldsymbol{A}}
\newcommand{\B}{\boldsymbol{B}}
\newcommand{\C}{\boldsymbol{C}}
\newcommand{\D}{\boldsymbol{D}}
\newcommand{\RR}{\mathbb{R}}
\newcommand{\NN}{\mathbb{N}}
\newcommand{\ZZ}{\mathbb{Z}}
\newcommand{\CC}{\mathbb{C}}
\newcommand{\de}{\operatorname{deg}}
\newcommand{\di}{\operatorname{dim}}
\newcommand{\E}{\mathcal{E}}

\newcommand{\p}{\pi}
\newcommand{\q}{\sigma}

\newcommand{\stir}{\genfrac{\{}{\}}{0pt}{}}
\newcommand{\stiri}{\genfrac{[}{]}{0pt}{}}
\newcommand{\fp}{\mathscr{P}}
\newcommand{\G}{\mathscr{G}}

\newcommand{\cb}{\mathbf{c}}
\newcommand{\ab}{\mathbf{a}}
\newcommand{\bb}{\mathbf{b}}

\newtheorem{thm}{Theorem}
\newtheorem{lem}[thm]{Lemma}
\newtheorem{cor}[thm]{Corollary}

\newtheorem{prop}[thm]{Proposition}
\newtheorem{Rem}[thm]{Remark}
\newtheorem{Def}[thm]{Definition}
\newtheorem{ex}[thm]{Example}
\newtheorem{alg}[thm]{Algorithm}
\date{}

\title{Stirling numbers and Gregory coefficients for the factorization of Hermite subdivision operators}
\author{
{Caroline Moosm\"uller}\thanks{Department of Mathematics, University of California, San Diego, 9500 Gilman Drive, La Jolla, CA 92093, USA. \texttt{cmoosmueller@ucsd.edu} (corresponding author)} 
\and 
{Svenja H\"uning}\thanks{ Institute of Geometry, TU Graz, Kopernikusgasse 24, 8010 Graz, Austria. \texttt{huening@tugraz.at}}
\and
{Costanza Conti}\thanks{DIEF, Universit\`{a} di Firenze, Viale Morgagni 40/44, 50134 Firenze, Italy. \texttt{costanza.conti@unifi.it}}}

\begin{document}

\maketitle

\begin{abstract}
In this paper we present a factorization framework for Hermite subdivision schemes refining function values and first derivatives, which satisfy a spectral condition of high order. In particular we show that spectral order $d$ allows for $d$ factorizations of the subdivision operator with respect to the \emph{Gregory operators}: A new sequence of operators we define using Stirling numbers and Gregory coefficients. We further prove that the $d$-th factorization provides a ``convergence from contractivity'' method for showing $C^d$-convergence of the associated Hermite subdivision scheme.
The power of our factorization framework lies in the reduction of computational effort for large $d$: In order to prove $C^d$-convergence, up to now, $d$ factorization steps were needed, while our method requires only one step, independently of $d$. Furthermore, in this paper, we show by an example that the spectral condition is not equivalent to the reproduction of polynomials.

 
\end{abstract}

\section{Introduction}

Hermite subdivision schemes are iterative refinement rules, which, applied to discrete vector data, produce a function and its consecutive derivatives in the limit. They find similar applications as classical subdivision schemes \cite{cavaretta91}, but are preferred when the modeling of first derivatives (or even higher derivatives) is of particular interest. This can be the case, for example, for the generation of curves and surfaces \cite{dyn02,merrien92,xue05,xue06}, for the construction of multiwavelets \cite{cotronei17,cotronei18}, for interpolating and approximating manifold-valued tangent vector data \cite{moosmueller16,moosmueller17}, and for the analysis of biomedical images
\cite{conti15,uhlmann14}.

The convergence of subdivision schemes as well as the analysis of the regularity of their limit functions are topics of high interest. It is well-known that such analyses are strongly connected to the factorization of the associated subdivision operator
\cite{cohen96,conti16,dyn02,merrien12,micchelli98}.

In this paper we study factorization properties of subdivision operators $S_{\A}: \ell(\mathbb{Z})^2 \rightarrow \ell(\mathbb{Z})^2$, which correspond to Hermite subdivision schemes producing functions and first derivatives:
\begin{equation}\label{eq:intro_sd}
\left(S_{\A}\cb\right)_j=\sum_{k \in \ZZ}\A_{j-2k}\cb_j, \quad j \in \ZZ.
\end{equation}
Here $\cb$ is a sequence of $2$-dimensional vectors (the input data), and $\A$ is a finitely-supported sequence of $(2 \times 2)$-matrices, called the \emph{mask} of the operator. 

We prove that every Hermite subdivision operator \eqref{eq:intro_sd} satisfying the \emph{spectral condition of order $d$} (\Cref{def:spectral}), can be factorized with respect to the operators $\G^{[n]}:\ell(\mathbb{Z})^2 \rightarrow \ell(\mathbb{Z})^2$ defined by
\begin{equation}\label{def:gregory_op}
\G^{[n]}=\left[\begin{array}{cc}
		0 & \Delta^n \\
        \Delta & -\sum_{\ell=0}^{n-1}G_{\ell}\Delta^{\ell}
		\end{array}\right], \quad n=1,\ldots,d,
\end{equation}
with $\Delta$ the forward difference operator and with the understanding that $\Delta^0=\operatorname{id}$ and $\Delta^n=\Delta(\Delta^{n-1})$. By $G_n$ we denote the \emph{Gregory coefficients}, which are a well studied sequence in number theory, see e.g.\ \cite{candelpergher12,kowalenko10,merlini06}. They can be computed from the Stirling numbers of the first kind; see \Cref{fig:greg} for the first Gregory coefficients $G_n$, $n=0,\dots,6$.
We call $\G^{[n]}$ the \emph{$n$-th Gregory operator}.
\begin{table}
\centering
\begin{tabular}{c|c|c|c|c|c|c|c}
 n&0 &1 &2 &3&4&5&6 \\[0.1cm] \hline 
$G_n$& 1 & 1/2& -1/12& 1/24 & -19/720 & 3/160 & -863/60480\\
\end{tabular}
\caption{First few Gregory coefficients $G_n$.}
\label{fig:greg}
\end{table}
The main results of this paper, proved in \Cref{sec:proof}, are
\begin{thm}[Main result]\label{intro:main}
Let $S_{\A}$ be a subdivision operator \eqref{eq:intro_sd} satisfying the spectral condition of order $d\geq 1$. Then for $n=1,\ldots, d$ there exist subdivision operators $S_{\B^{[n]}}$ such that
\begin{equation}\label{nth_greg}
\G^{[n]} S_{\A}=2^{-n}S_{\B^{[n]}}\G^{[n]}, \quad n=1,\ldots, d.
\end{equation}
\end{thm}

We in addition show that the last factorization gives rise to an easy-to-check condition for the $C^d$-convergence of Hermite subdivision schemes:
\begin{cor}\label{intro:cor_main}
With notation as in \Cref{intro:main},
if $S_{\B^{[d]}}$ is contractive, then the Hermite subdivision scheme associated with $S_{\A}$ is $C^d$-convergent.
\end{cor}
Furthermore, in \Cref{sec:ex}, we show that for primal schemes the spectral condition of order $d$ does not imply that polynomials up to degree $d$ are reproduced, while it is known that the reverse implication holds true \cite{conti14}. Up to now, these two concepts were conjectured to be equivalent.
\subsection{Impact of our results}
Factorization of subdivision operators for proving convergence/regularity of the associated subdivision scheme is a standard method in scalar subdivision \cite{dyn92}, vector subdivision \cite{charina05,cohen96,sauer02} and Hermite subdivision \cite{conti16,conti17,merrien12}. Nevertheless, the results for Hermite subdivision schemes are only concerned with factorizing \emph{once}, that is, with proving the \emph{minimal} regularity of the scheme (for example, in our case the minimal regularity is $1$ since we consider schemes dealing with function values and first derivatives), see e.g.\ \cite{merrien12}. Many authors, however, are interested in higher regularity than the minimal one \cite{conti14,han05,jeong17,moosmueller18}. We show in this paper that for Hermite schemes, the Gregory operators provide the necessary factorization tool to prove regularity higher than $1$.

It is worthwhile noting that \emph{every} Hermite subdivision operator satisfying the spectral condition of order $d$ can be factorized with respect to the Gregory operators. In general, for such an Hermite subdivision operator, there exist infinitely many possibilities to factorize \emph{beyond} the Taylor factorization (i.e.\ to prove regularity higher than the minimal one). This is due to the theory of factorizing vector schemes \cite{charina05,cohen96,sauer02}, which involves choosing an eigenvector of the vector subdivision operator, and completing this vector to a basis of $\RR^2$ (obviously, there are infinitely many ways to do this). Moreover, the choice of an eigenvector for the $(k+1)$-th factorization depends on the $k$-th factorization. This means that one can only factorize step-by-step, which drastically slows down computations. It also means that different Hermite schemes factorize with respect to very different operators. These facts can be seen from the computations in \cite{conti14,jeong17}.

We prove that the spectral condition guarantees the existence of \emph{one} factorization that works for \emph{all} Hermite subdivision operators. The key to this factorization is a clever choice of eigenvectors.

We would like to stress the improvement for computations arising from the Gregory factorization. In order to prove that a Hermite scheme is $C^d$-convergent, $d\geq 1$, up to now, $d$ factorization steps were necessary, see again \cite{conti14,jeong17}. As shown in \Cref{intro:cor_main}, we reduce this procedure to one single factorization: $n=d$ in \eqref{def:gregory_op} provides the operator with respect to which one has to factorize.

We mention that for $d=1$ the complete Taylor operator \cite{merrien12} and $\G^{[1]}$ provide the same tool for proving $C^1$-convergence for schemes. In this sense
the Gregory operators are direct extensions of the complete Taylor operator of dimension $2$. However, the Taylor operator is more powerful in proving the minimal regularity of a scheme, as it also works for schemes of general dimension $k, k\geq 2$, and for multivariate schemes. We thus consider the Gregory operators as a first step towards an extension of the Taylor operator for proving higher regularity than the minimal one.

Since the Stirling numbers and the Gregory coefficients are closely connected to higher-order finite differences, it is not too surprising that they appear in our construction. Nevertheless, we find it remarkable that the Gregory coefficients appear in such a natural manner and allow for a complete and easy description of the operators \eqref{def:gregory_op}.

\subsection{Organization of the paper}
The paper is organized as follows: The preliminary section (\Cref{sec:pre}) fixes the notation and recalls basic facts about subdivision schemes, factorization of subdivision operators, and the convergence of vector and Hermite schemes. \Cref{sec:forward_differences} introduces Stirling numbers and Gregory coefficients and discusses a recursion involving iterated forward differences. The main results are stated and proved in \Cref{sec:proof}. Examples of the Gregory factorization and of its use are provided in \Cref{sec:ex}. In this section we also show that the spectral condition does not imply the reproduction of polynomials.  \Cref{sec:conclusion} concludes the paper.

\section{Preliminaries}\label{sec:pre}
\subsection{Hermite subdivision schemes}\label{subsec:Hermite}

We denote by $\ell(\ZZ)^2$ the space of $\RR^2$-valued sequences $\cb=\left(\cb_j: j \in \ZZ\right)$, and by $\ell(\ZZ)^2_{\infty}$ the space of $\RR^2$-valued sequences with finite infinity-norm:
\begin{equation*}
\|\cb\|_{\infty}:=\sup_{j \in \ZZ}|\cb_j|_{\infty}<\infty,
\end{equation*}
where $|\cdot|_{\infty}$ is the infinity-norm on $\RR^2$.
Similarly, we define the space $\ell(\ZZ)^{2\times 2}$ of matrix-valued sequences $\A=\left(\A_j: j \in \ZZ \right)$, and the space $\ell(\ZZ)^{2\times 2}_{\infty}$ of all such sequences with finite infinity-norm:
\begin{equation*}
\|\A\|_{\infty}:=\sup_{j \in \ZZ}|\A_j|_{\infty}<\infty,
\end{equation*}
where $|\cdot|_{\infty}$ is the operator norm for matrices in $\RR^{2 \times 2}$ induced by the infinity-norm on $\RR^2$. We also consider the spaces $\ell(\ZZ)^{2}_{0}$ and $\ell(\ZZ)^{2\times 2}_{0}$ which consist of finitely supported vector resp. matrix sequences.
\begin{Def}[Subdivision operator]\label{def:subd_op}
A \emph{subdivision operator} with \emph{mask} $\A \in \ell(\ZZ)^{2\times 2}_{0}$ is the map $S_{\A}:\ell(\ZZ)^{2} \to \ell(\ZZ)^{2}$ defined by
\begin{equation}\label{eq:sdo}
\left(S_{\A}\cb\right)_j=\sum_{k \in \ZZ}\A_{j-2k}\cb_j, \quad \cb \in \ell(\ZZ)^{2},\, j \in \ZZ.
\end{equation}
\end{Def}
Note that due to the finite support of the mask $\A$, the sum in \Cref{def:subd_op} is finite. Furthermore, if $\cb \in \ell(\ZZ)^{2}_{\infty}$, then $S_{\A}\cb \in \ell(\ZZ)^{2}_{\infty}$. Therefore we can define the norm of a subdivision operator $S_{\A}$ by
\begin{equation*}
\|S_{\A}\|_{\infty}=\sup\{ \|S_{\A}\cb\|_{\infty}: \|\cb\|_{\infty}=1\}.
\end{equation*}
\begin{Def}[Hermite subdivision scheme]\label{def:hermite}
Let $S_{\A}$ be a subdivision operator \eqref{eq:sdo}. An \emph{Hermite subdivision scheme} is the iterative procedure of constructing vector-valued sequences by
\begin{equation}\label{eq:Hermite_sds}
\D^{n+1}\cb^{[n+1]}=S_{\A}\D^{n}\cb^{[n]},\quad n\in \NN,
\end{equation}
from initial data $\cb^{[0]} \in \ell(\ZZ)^{2}$. Here $\D$ denotes the diagonal matrix $\D=\operatorname{diag}\left(1, 1/2\right)$.
\end{Def}
\begin{Def}[Convergence of Hermite subdivision schemes]\label{def:convergent_hermite}
An Hermite subdivision scheme is $C^{d}$-convergent, $d \geq 1$, if for every input data $\cb^{[0]}\in \ell(\ZZ)^{2}_{\infty}$ and any compact $K \subset \mathbb{R}$, there exists a function $\varphi \in C^d(\RR)$ such that $\Phi=[\varphi,\varphi']^T: \RR \to \RR^2$ and the sequence $\cb^{[n]}$ defined by \eqref{eq:Hermite_sds} satisfies
\begin{equation}\label{eq:Hermite_convergence}
\lim_{n \to \infty}\sup_{j\in \ZZ\cap K}|\cb^{[n]}_j-\Phi\left(2^{-n}j\right)|_{\infty}=0.
\end{equation}
Furthermore, we request that there exists at least one $\cb^{[0]} \in \ell(\ZZ)_{\infty}^{2}$ such that $\varphi \neq 0$.
\end{Def}
The regularity of Hermite schemes is studied in many papers, see e.g.\ \cite{conti17,dubuc06,dubuc05,dyn95,dyn99,merrien12}.
Note that these papers are concerned with the \emph{minimal} regularity of an Hermite subdivision scheme (e.g.\ with regularity $1$).
Along the lines of \cite{conti14,han05,jeong17,moosmueller18}, we are interested in the regularity which is higher than one.

For a sequence $\cb$ we define the forward difference operator by
\begin{equation}\label{def:seq_forward_diff}
\left(\Delta \cb\right)_j=\cb_{j+1}-\cb_j, \quad j \in \ZZ.
\end{equation}
In analogy to \eqref{def:seq_forward_diff}, we define the forward difference operator for functions $f$ by
\begin{align}\label{def:delta}
\left(\Delta f\right)(x)=f(x +1)-f(x), \quad x \in \RR.
\end{align}
If $f$ is differentiable, we define the differential operator
\begin{align}\label{def:D}
 Df=f',
\end{align}
where we take the derivative component-wise.
By sampling $f$ on $\ZZ$, we obtain a vector-valued sequence $\cb_f=(f(j): j \in \ZZ)$. Since in this paper we are only concerned with sampled functions, we denote the sequence $\cb_f$ again by $f$. Therefore, by $S_{\A}f=g$ we mean $S_{\A}\cb_f=\cb_g$ for two functions $f,g$.
Note that this notation is consistent with the forward difference operators for functions and sequences:
\begin{equation}
(\Delta \cb_f)_j=(\Delta f)(j), \quad j \in \ZZ.
\end{equation}

We denote by $\Pi_k$ the set of polynomials with real coefficients of degree $\leq k$, $k \geq 0$. If $\pi \in \Pi_k$, then we write
\begin{equation*}
\pi(x)=\sum_{j=0}^k \pi[j]x^j,
\end{equation*}
that is, we denote the $j$-th coefficient of $\pi$ by $\pi[j] \in \RR$, $j=0,\ldots,k$.
\begin{Def}[Spectral condition]\label{def:spectral}
A subdivision operator $S_{\A}$ satisfies the \emph{spectral condition of order $d$}, $d\geq 1$, if there exist polynomials
$\fp_k \in \Pi_k$, $\fp_k[k]=1/k!$, such that
\begin{equation}\label{eq:spectral}
 S_{\A}\left[ \begin{array}{c}
               \fp_k\\
               D\fp_k
              \end{array}
	\right]
= 2^{-k}\left[ \begin{array}{c}
               \fp_k\\
               D\fp_k
              \end{array}
	\right],
\end{equation}
$k=0,\ldots, d$.
A subdivision operator satisfying the spectral condition of order $d$ is called \emph{Hermite subdivision operator of spectral order $d$}. The polynomials $\fp_k, k=0,\ldots, d$, are named \emph{spectral polynomials of $S_{\A}$}.
\end{Def}
\begin{Def}\label{def:reprod}
Let $S_{\A}$ be a subdivision operator. The Hermite subdivision scheme associated with $S_{\A}$ is said to \emph{reproduce a function $f\in C^{1}(\RR)$} if for initial data $\cb^{[0]}_j=[f(j),f'(j)]^T$, the iterated sequence $\cb^{[n]}$ defined by \eqref{eq:Hermite_sds} is given by $\cb^{[n]}_j=[f(2^{-n}j),f'(2^{-n}j)]^T$, $j\in \ZZ, n\geq 1$.
\end{Def}
The spectral condition was first introduced by \cite{dubuc09}, see also \cite{conti14,merrien12}.
In \cite{dubuc09} it is proved that the spectral condition is equivalent to a special \emph{sum rule} introduced by \cite{han03b,han05}. 
Note that in \Cref{def:reprod} we use the \emph{primal} parametrization, as opposed to dual or more general parametrizations which can be considered, see e.g.\ \cite{conti18,conti14}.
Furthermore, \cite{conti14} shows that reproduction of $\Pi_d$ implies the spectral condition of order $d$. 

The reverse implication was not yet clear, but we here put into evidence that it is actually false. Indeed, the primal Hermite scheme in \Cref{ex:H1} satisfies the spectral condition of order $d=4$ (for $\theta=1/32$), but polynomials of degree $4$ are not reproduced.

We mention that the spectral condition is a crucial property for the factorizability of an Hermite subdivision operator, but, as proved in \cite{merrien11,merrien18}, it is not necessary for convergence.

\subsection{Factorization of subdivision operators}
In order to discuss factorizations of Hermite subdivision operators and their connection to regularity higher than the minimal, we have to introduce vector subdivision schemes. The following part on vector subdivision schemes presented here is simplified and an adapted version of constructions and results from the general theory of vector subdivision schemes, see e.g.\ \cite{charina05,cohen96,micchelli98,sauer02} for details.
\begin{Def}[Vector subdivision scheme]
Let $S_{\B}$ be a subdivision operator \eqref{eq:sdo}. A \emph{vector subdivision scheme} is the iterative procedure of constructing vector-valued sequences by
\begin{equation}\label{eq:vector_sd}
\cb^{[n+1]}=S_{\B}\cb^{[n]},\quad n\in \NN,
\end{equation}
from initial data $\cb^{[0]}\in \ell(\ZZ)^2$.
\end{Def}
Note that an Hermite subdivision scheme is a \emph{level-dependent} case of vector subdivision, i.e.\ it can be generated by applying vector subdivision operators that vary with the level $n$: $S_{\B^{[n]}}=\D^{-(n+1)}S_{\A}D^{n}$. The crucial difference between Hermite and vector subdivision schemes lies in the definition of convergence:
\begin{Def}[Convergence of vector subdivision schemes]
A vector subdivision scheme is $C^{d}$-convergent, $d \geq 0$, if for every input data $\cb^{[0]}\in \ell(\ZZ)^{2}_{\infty}$ and any compact $K \subset \mathbb{R}$, there exists a vector-valued function $\Psi\in C^d(\RR,\RR^2)$ such that the sequence $\cb^{[n]}$ defined by \eqref{eq:vector_sd} satisfies
\begin{equation}\label{eq:vector_convergence}
\lim_{n \to \infty}\sup_{j\in \ZZ\cap K}|\cb^{[n]}_j-\Psi\left(2^{-n}j\right)|_{\infty}=0,
\end{equation}
and there exists at least one $\cb^{[0]} \in \ell(\ZZ)^{2}_{\infty}$ such that $\Psi \neq 0$. $C^0$-convergent vector schemes are simply called ``convergent''.
\end{Def}
Following \cite{micchelli98}, for a mask $\B$, we define $\E_{\B}$ by
\begin{equation*}
 \E_{\B}=\{v \in \RR^2: \sum_{j\in \ZZ}\B_{2j}v=v,\, \sum_{j\in \ZZ}\B_{2j+1}v=v\}.
\end{equation*}
It is well-known that the convergence of the vector subdivision scheme associated with $S_{\B}$ implies that there exists $v \neq 0$ such that $v\in\E_{\B}$.
The following is clear from the definition of $\E_{\B}$:
\begin{lem}
 Let $\B$ be a mask. Let $v \in \RR^2$. Then the following are equivalent:
 \begin{enumerate}
  \item $v \in \E_{\B}$,
  \item $S_{\B}v=v$.
 \end{enumerate}
 As in \Cref{subsec:Hermite}, we identify the constant function $v$ with the constant sequence $\cb_v=\left(v: j\in \ZZ\right)$.
\end{lem}
Therefore the space $\E_{\B}$ is the space of all eigenvectors (constant sequences) of $S_{\B}$ with respect to the eigenvalue $1$.

In this paper we are only concerned with masks $\B$ with $\dim \E_{\B}=1$. Following \cite{charina05}, we call a matrix $V$ an \emph{$\E_{\B}$-generator}, if $V=[v,w]$, where $v\neq 0$ spans $\E_{\B}$, and $v$ and $w$ are linearly independent.

We now introduce a generalization of the forward difference operator $\Delta$ for vector schemes. Let $V$ be an invertible matrix. Define $\Delta_V$ by
\begin{equation}\label{eq:delta_V}
 \Delta_V=\left[
	  \begin{array}{cc}
	   \Delta & 0 \\
	   0 & 1
	  \end{array}
	\right]V^{-1}.
\end{equation}
In \cite{charina05}, the matrix $V$ is assumed to be orthogonal. We choose a slightly more general approach, which, however, does not change the validity of the results below.
From \cite{charina05,micchelli98,sauer02} we have the following result concerning the factorization of subdivision operators:
\begin{thm}\label{thm:vector}
 Let $S_{\B}$ be a subdivision operator \eqref{eq:sdo} and assume that $\di \E_{\B}=1$. For an $\E_{\B}$-generator $V$ there
 exists a subdivision operator $S_{\C}$ such that
 \begin{equation*}
  \Delta_{V}S_{\B}=2^{-1}S_{\C}\Delta_{V}.
 \end{equation*}
Furthermore $\di\E_{\C}=1$ or $\E_{\C}=\{0\}$.
\end{thm}

From \cite[Corollaries 5 and 8]{charina05} we obtain
\begin{thm}\label{thm:converge_vector}
With assumption and notation as in \Cref{thm:vector}, we have
\begin{enumerate}
\item\label{it:converge} If $\|(2^{-1}S_{\C})^N\|_{\infty}<1$, for some $N\geq 1$, that is, if $2^{-1}S_{\C}$ is \emph{contractive}, then the vector subdivision scheme associated with $S_{\B}$ is convergent.
\item\label{it:smooth} If the vector scheme associated with $S_{\C}$ is $C^{d}$-convergent, then the vector scheme associated with $S_{\B}$ is $C^{d+1}$-convergent, $d\geq 0$.
\end{enumerate}
\end{thm}
Note that \cite{charina05} shows stronger results than the ones mentioned above. We only need these special cases.
Furthermore, in part \ref{it:smooth} of the theorem we dropped the assumption that $S_{\B}$ is convergent. This is possible due to the following reason: If $S_{\C}$ is convergent, then $2^{-1}S_{\C}$ is contractive, and thus by part \ref{it:converge} of the theorem, $S_{\B}$ converges.

Note that in order to show $C^{d}$-convergence of a vector subdivision scheme, there are infinitely many ways to factorize $S_{\B}$, respectively to obtain an operator $S_{\C}$. Nevertheless, in \cite{charina05} it is shown that if $2^{-1}S_{\C}$, coming from a factorization with respect to a matrix $V$, is contractive, then the operator $2^{-1}S_{\mathbf{E}}$ obtained from any other valid factorization is also contractive. Therefore, the choice of $V$ is irrelevant for proving convergence from contractivity.

As in \cite{charina05}, we may iterate \Cref{thm:vector}, to conclude the following: 

\begin{lem}\label{lem:iteration_vector}
Let $S_{\B}$ be a subdivision operator with $\dim\E_{\B}=1$. Let $\C^{[0]}=\B$ and let $V^{[0]}$ be an $\E_{\C^{[0]}}$-generator. If for $k\geq 0$ and $n=1,\ldots,k+1$ there exist matrices $V^{[n]}$ and masks ${\C^{[n]}}$ such that $V^{[n]}$ is an $\E_{\C^{[n]}}$-generator and such that
\begin{equation}\label{eq:iteration_vector}
\Delta_{V^{[n]}}\cdots \Delta_{V^{[0]}}S_{\B}=2^{-(n+1)}S_{\C^{[n+1]}}\Delta_{V^{[n]}}\cdots \Delta_{V^{[0]}}, \quad n=0,\ldots,k,
\end{equation}
and $2^{-1}S_{\C^{[k+1]}}$ is contractive, then the vector scheme associated with $S_{\B}$ is $C^{k}$-convergent.
\end{lem}
Our construction of the $n$-th Gregory operator \eqref{def:gregory_op} for Hermite subdivision operators relies heavily on the iteration \eqref{eq:iteration_vector}.

We now continue with Hermite subdivision operators. Denote by $T$ the \emph{Taylor operator} of dimension 2,
\begin{equation*}
 T=\left[
 \begin{array}{cr}
  \Delta & -1 \\
  0 & 1
 \end{array}
  \right],
\end{equation*}
which was first defined in \cite{merrien12} for the convergence and smoothness analysis of Hermite schemes. We have the following results from \cite{merrien12}:
\begin{thm}\label{thm:Taylor}
 If $S_{\A}$ is an Hermite subdivision operator of spectral order at least $1$, then there exists a subdivision operator $S_{\B}$ such that
 \begin{equation*}
  TS_{\A}=2^{-1}S_{\B}T.
 \end{equation*}
 Also, $\E_{\B}$ is spanned by $[0,1]^T$. In particular, $\dim\E_{\B}=1$.
\end{thm}
If $\E_{\B}$ is spanned by $[0,1]^T$, the vector scheme associated with $S_{\B}$ has limit functions of the form $\Psi=[0,\psi_1]^T$ for all input data (this follows from results in \cite{micchelli98}; an explicit proof can also be found in \cite{moosmueller18}). Combining this with results from \cite{merrien12} we obtain
\begin{thm}\label{thm:convergent_Hermite}
Let $S_{\A}$ be an Hermite subdivision operator of spectral order $d$, $d\geq 1$, and let $S_{\B}$ be as in \Cref{thm:Taylor}. If the \emph{vector subdivision scheme} associated with $S_{\B}$ is $C^k$-convergent, $k\geq 0$, then the Hermite subdivision scheme associated with $S_{\A}$ is $C^{k+1}$-convergent. 
\end{thm}
We mention that this theorem is also stated in \cite{conti14}. From \Cref{thm:convergent_Hermite} we see that a tool for checking $C^k$-convergence of a vector subdivision scheme is needed.
Combining \Cref{thm:convergent_Hermite} with \Cref{lem:iteration_vector}, we can state:
\begin{lem}\label{lem:iteration_hermite}
Let $d\geq 1$ and let $S_{\A}$ be a subdivision operator.
Suppose that for $n=0,\ldots,d$, there exist matrices $V^{[n]}$ and masks ${\C^{[n]}}$, such that $V^{[n]}$ is an $\E_{\C^{[n]}}$-generator and such that
\begin{align*}
&TS_{\A}=2^{-1}S_{\C^{[0]}}T\\
&\Delta_{V^{[n-1]}}\cdots \Delta_{V^{[0]}}TS_{\A}=2^{-(n+1)}S_{\C^{[n]}}\Delta_{V^{[n-1]}}\cdots \Delta_{V^{[0]}}T, \quad n=1,\ldots,d.
\end{align*}
If $2^{-1}S_{\C^{[d]}}$ is contractive, then the Hermite subdivision scheme associated with $S_{\A}$ is $C^{d}$-convergent.
\end{lem}
In \Cref{sec:proof} we prove that the spectral condition of order $d$ stated by equation \eqref{eq:spectral} implies the existence of a factorization as in \Cref{lem:iteration_hermite}. The matrices $V^{[n]}$ are given by
\begin{equation*}
  V^{[0]}=\left[
	    \begin{array}{cc}
	     0 & 1 \\
	     1 & 0
	    \end{array}
	  \right], \quad
  V^{[n]}=\left[
	    \begin{array}{cc}
	     1 & 0 \\
	     G_n & 1
	    \end{array}
	  \right], \quad n \geq 1,
\end{equation*}
where $G_n$ are the Gregory coefficients, see the next section.

\section{A recursion involving iterated forward differences}\label{sec:forward_differences}
\subsection{Stirling numbers and Gregory coefficients}
%
We do not attempt to give an overview of properties and results concerning the Stirling numbers, as they are fundamental sequences in number theory, but we cite \cite{graham94} for an introduction.
We summarize a few properties relevant for this paper, which are all taken from \cite{graham94}.

The \emph{Stirling numbers of the first kind}, denoted by $\stiri{n}{m}$, count the numbers of ways to arrange $n$ elements into $m$ cycles. From the initial conditions
\begin{equation*}
\stiri{0}{0}=1, \quad \stiri{n}{0}=\stiri{0}{n}=0, \quad n\geq 1,
\end{equation*}
they can be computed via the following recurrence relation:
\begin{equation*}
\stiri{n+1}{m}=n\stiri{n}{m}+\stiri{n}{m-1}, \quad m\geq 1.
\end{equation*}

The \emph{Stirling numbers of the second kind}, denoted by $\stir{n}{m}$, count the number of ways to split a set of $n$ elements into $m$ non-empty subsets.
They can be computed using Binomial coefficients:
\begin{equation}\label{def:stirling2}
\stir{n}{m}=\frac{1}{m!}\sum_{j=0}^m {m \choose j} (-1)^{m-j}j^n.
\end{equation}
We further need the following properties:
\begin{equation}\label{prop:stir}
\stir{n}{n}=\stir{n}{1}=1,\quad \stir{n}{m}=0 \text{ if } m> n.
\end{equation}

We introduce the \emph{Gregory coefficients}
\begin{equation}\label{def:greg}
G_n=\frac{1}{n!}\sum_{j=0}^{n}\stiri{n}{j}\frac{(-1)^{n-j}}{j+1},
\end{equation}
which are also known as the \emph{Cauchy numbers of the first kind}, the \emph{Bernoulli numbers of the second kind} and the \emph{reciprocal logarithmic numbers}, see e.g.\ \cite{blagouchine17,kowalenko10,merlini06}.
The Gregory coefficients appear in many interesting contexts: As the coefficients in a power series expansion of the reciprocal logarithm \cite{kowalenko10}, in Gregory's method for numerical integration \cite{phillips72}, in various series representation involving Euler's constant \cite{blagouchine16,candelpergher12}, and in a series expansion of the Gompertz constant \cite{mezo14}, to name a few.

Our results in \Cref{sec:proof} are based on the following relation between the Gregory coefficients and the Stirling numbers of the second kind:
\begin{equation}\label{relation_stir_greg}
\sum_{j=0}^n \stir{n}{j}j!\,G_j=\frac{1}{n+1}.
\end{equation}
This is proved in e.g.\ \cite{merlini06}. Note that \cite{merlini06} shows \eqref{relation_stir_greg} for $\mathscr{C}_n=n!\, G_n$, and they call $\mathscr{C}_n$ the Cauchy numbers of the first kind.

\subsection{Iterated forward differences}
%
In this section we collect properties concerning the operators $\Delta$ and $D$, as well as the iterates $\Delta^{\ell}$, $\ell \geq 1$.
They can be derived easily from the respective definitions; for the convenience of the reader, we prove some of them.

The following lemma is clear from the definitions \eqref{def:delta} and \eqref{def:D}.
\begin{lem}
The differential operator $D$ and the forward difference operator $\Delta$ defined in \eqref{def:D} and \eqref{def:delta} commute:
 \begin{equation}
  \Delta D = D \Delta.
 \end{equation}
\end{lem}
\begin{lem}\label{lem:D_Delta_coeff}
For $k\geq 1$, both the differential operator $D$ and the forward difference operator $\Delta$ map $\Pi_k$ to $\Pi_{k-1}$. For $\pi \in \Pi_k$, the coefficients of $D\pi$ resp. $\Delta\pi$ are given by
\begin{align*}
&(D\pi)[j]=(j+1)\pi[j+1], \\
&(\Delta \pi)[j]=\sum_{m=j}^{k-1} {m+1 \choose j}\pi[m+1],
\end{align*}
$j=0,\ldots,k-1$. For $\pi \in \Pi_0, D\pi=\Delta \pi=0$.
\end{lem}
\begin{proof}
We prove the part involving $\Delta$, the rest is clear. For $\pi \in \Pi_k$, we obtain
\begin{align*}
 (\Delta \pi)(x)&=\pi(x+1)-\p(x)=\sum_{\ell=0}^{k}\p[\ell]\left((x+1)^{\ell}-x^{\ell}\right)\\
 &=\sum_{\ell=0}^{k}\p[\ell]\left(\sum_{m=0}^{\ell}{\ell \choose m}x^m-x^{\ell}\right)\\
 &=\sum_{\ell=1}^{k}\p[\ell]\left(\sum_{m=0}^{\ell-1}{\ell \choose m}x^m\right)
 =\sum_{\ell=0}^{k-1}\p[\ell+1]\left(\sum_{m=0}^{\ell}{\ell+1 \choose m}x^m\right)\\
 &=\sum_{m=0}^{k-1}\sum_{\ell=m}^{k-1}\p[\ell+1]{\ell+1 \choose m}x^m.
\end{align*}
This shows that $\Delta \pi \in \Pi_{k-1}$ and verifies the formula for the coefficients as stated in the Lemma.
\end{proof}
\begin{cor}\label{cor:Delta_minus_D}
For a polynomial $\pi\in \Pi_k$ with $k\geq 2$, the polynomial $(\Delta-D)\pi$ has degree $k-2$ and its coefficients are given by
\begin{equation*}
(\Delta-D)\pi[j]=\sum_{m=j+1}^{k-1}{m+1 \choose j}\pi[m+1], \quad j=0,\ldots,k-2.
\end{equation*}
For $\pi \in \Pi_0$ or $\pi\in\Pi_1$, $(\Delta-D)\pi=0$.
\end{cor}
\begin{lem}\label{lem:iterated_Delta}
For $\pi \in \Pi_k$, $k\geq1$ and $1\leq \ell\leq k$ we have
 \begin{equation*}
   \frac{1}{\ell!}\,\Delta^{\ell}\pi(x)=\sum_{j=0}^{k-\ell}\sum_{m=j}^{k}\pi[m]{m \choose j}
   \stir{m-j}{\ell}x^{j}.
 \end{equation*}
If $k=0$ or $\ell > k$, then $\Delta^{\ell}\pi=0$.
\end{lem}
\begin{proof}
The cases $k=0$ and $\ell > k$ are clear. For $k\geq 1$ and $1\leq \ell\leq k$ we use the following well-known formula (see e.g.\ \cite[p.\ 188]{graham94}): 
\begin{equation*}
 (\Delta^{\ell}\pi)(x)=\sum_{s=0}^{\ell}{\ell \choose s}(-1)^{\ell-s}\p(x+s).
\end{equation*}
Then by \eqref{def:stirling2} we obtain
\begin{align*}
 (\Delta^{\ell}\pi)(x)&=\sum_{s=0}^{\ell}{\ell \choose s}(-1)^{\ell-s}\sum_{m=0}^k\pi[m](x+s)^m\\
 &=\sum_{s=0}^{\ell}{\ell \choose s}(-1)^{\ell-s}\sum_{m=0}^k\pi[m]\sum_{j=0}^{m}{m \choose j} x^{j}s^{m-j}\\
 &=\sum_{j=0}^k\sum_{m=j}^{k}\ell! \,\pi[m]{m \choose j} \stir{m-j}{\ell}x^{j}.
\end{align*}
Since the degree of $\Delta^{\ell}\pi$ is $k-\ell$, we obtain the result
\begin{equation*}
(\Delta^{\ell}\pi)(x)=\sum_{j=0}^{k-\ell}\sum_{m=j}^{k}\ell! \,\pi[m]{m \choose j} \stir{m-j}{\ell}x^{j}.
\end{equation*}
Note that the vanishing of $(\Delta^{\ell}\pi)[j]$
for $j=k-\ell+1,\ldots,k$ also follows from \eqref{prop:stir}: For $m=j,\ldots, k$, the value $m-j \in \{0,\ldots,\ell-1\}$, and therefore in these cases $\stir{m-j}{\ell}=0$.
\end{proof}
%

\subsection{Solving a recursion with iterated forward differences}
%
In this section we set the basis for the results in \Cref{sec:proof}. We solve the recursion of iteratively applying operators of the from \eqref{eq:delta_V} to polynomials (more generally, functions).
\begin{lem}\label{lem:recursion_pq}
Let $f^{[1]}_k: \mathbb{R} \rightarrow \mathbb{R}$ and $g^{[1]}_k: \mathbb{R} \rightarrow \mathbb{R}$ ($k=0,1,\dots$) be two sequences of real-valued functions and let $(a_{n}, n\geq 1)$, be a sequence of real numbers.
We define invertible matrices by
 \begin{equation*}
  V^{[n]}=\left[
	    \begin{array}{cc}
	     1 & 0 \\
	     a_{n} & 1
	    \end{array}
	  \right], \quad n\geq 1
\end{equation*}
and the sequences of real-valued functions $(f^{[n+1]}_k, k \geq 0)$ and $(g^{[n+1]}_k:k \geq 0)$ by
\begin{align}\label{al:def_f_g}
\left[ \begin{array}{c}
               f_k^{[n+1]}\\
               g_k^{[n+1]}
              \end{array}
	\right]=
\Delta_{V^{[n]}}
\left[ \begin{array}{c}
               f_{k+1}^{[n]}\\
               g_{k+1}^{[n]}
              \end{array}
	\right], \quad n \geq 1,\, k\geq 0.
\end{align}
Then for $n \geq 1$ and $k \geq 0$ we have
\begin{align}\label{al:it_f_g}
 f_k^{[n+1]}&=\Delta^{n}f_{k+n}^{[1]},\\ \nonumber
 g_k^{[n+1]}&=g_{k+n}^{[1]}-\sum_{\ell=1}^{n}a_{\ell}\Delta^{\ell-1}f_{k+n}^{[1]},
\end{align}
with the understanding that $\Delta^0=\operatorname{id}$.
\end{lem}
\begin{proof}
We prove this claim by induction on $n$.
Observe from \eqref{eq:delta_V} that
 \begin{equation*}
  \Delta_{ V^{[n]}}=\left[
	    \begin{array}{cc}
	     \Delta & 0 \\
	     -a_{n}& 1
	    \end{array}
	  \right].	  
 \end{equation*}
Starting with $n=1$, from \eqref{al:def_f_g} we obtain
\begin{align*}
 f_k^{[2]}&=\Delta f_{k+1}^{[1]},\\
 g_k^{[2]}&=g_{k+1}^{[1]}-a_{1}f_{k+1}^{[1]},
\end{align*}
which is exactly \eqref{al:it_f_g} for $n=1$.
Assume that the statement is true for $n$, we prove it for $n+1$:
\begin{align*}
  f_k^{[n+1]}&=\Delta f_{k+1}^{[n]} = \Delta \Delta^{n-1}f_{k+1+n-1}^{[1]}=\Delta^{n}f_{k+n}^{[1]},\\
  g_k^{[n+1]}&=g_{k+1}^{[n]}-a_{n}f_{k+1}^{[n]}\\
		&=g_{k+1+n-1}^{[1]}-\sum_{\ell=1}^{n-1}a_{\ell}\Delta^{\ell-1}f_{k+1+n-1}^{[1]}
		-a_{n}\Delta^{n-1}f_{k+1+n-1}^{[1]}\\
		  &=g_{k+n}^{[1]}-\sum_{\ell=1}^{n-1}a_{\ell}\Delta^{\ell-1}f_{k+n}^{[1]}
		-a_{n}\Delta^{n-1}f_{k+n}^{[1]}\\
	      &=g_{k+n}^{[1]}-\sum_{\ell=1}^{n}a_{\ell}\Delta^{\ell-1}f_{k+n}^{[1]},
\end{align*}
which concludes the induction step.
\end{proof}
With a suitable choice of $(f^{[1]}_k, k \geq 0)$ and $(g^{[1]}_k:k \geq 0)$ from \Cref{lem:recursion_pq} we obtain the following
\begin{cor}\label{cor:pi_sig}
 Let $(h_k:k\geq 0)$ be a sequence of differentiable functions. Setting $f^{[1]}_k=\Delta D h_{k+2}$ and $g^{[1]}_k=\Delta h_{k+2}-Dh_{k+2}$, $k\geq 0$, in \Cref{lem:recursion_pq}, then with $a_{0}=1$
 \begin{align}\label{def:p}
 f_k^{[n]} &=\Delta^{n}Dh_{k+n+1}, \\ \label{def:q}
 g_k^{[n]} &=\Delta h_{k+n+1}-\sum_{\ell=0}^{n-1}a_{\ell}
 \Delta^{\ell}Dh_{k+n+1},
\end{align}
for $n\geq 1,\, k\geq 0$,
and with the understanding that $\Delta^0=\operatorname{id}$.
\end{cor}
We now consider the sequence \eqref{def:p} applied to polynomials.
\begin{prop}\label{prop:p}
Let $(\tau_k: k\geq 0)$ be a sequence of polynomials such that $\tau_k \in \Pi_k,\, k\geq 0$. For $n\geq 1$ and $k\geq 0$ define
 $\p_k^{[n]}=\Delta^{n}D\tau_{k+n+1}$. Then $\p_k^{[n]}\in \Pi_k$ and its coefficients are given by
 \begin{equation*}
  \frac{1}{n!}\, \p_k^{[n]}[j]=\sum_{m=j}^{k+n}(m+1)\tau_{k+n+1}[m+1]{m \choose j}
  \stir{m-j}{n}, \quad j=0,\ldots,k.
 \end{equation*}
\end{prop}
\begin{proof}
\Cref{lem:D_Delta_coeff} implies that the degree of $\pi_k^{[n]}$ is $\de(\tau_{k+n+1})-n-1=k$. From \Cref{lem:D_Delta_coeff} and \Cref{lem:iterated_Delta} we obtain the coefficients as stated in the Proposition.
\end{proof}

Similarly, we now consider the sequence \eqref{def:q} applied to polynomials.
\begin{prop}\label{prop:q}
Let $(\tau_k: k\geq 0)$ be a sequence of polynomials such that $\tau_k \in \Pi_k,\, k\geq 0$. For $n\geq 1, \, k \geq 0$, define
 $\q_k^{[n]}=\Delta \tau_{k+n+1}-\sum_{\ell=0}^{n-1}a_{\ell}
 \Delta^{\ell}D\tau_{k+n+1}$, where $(a_n:n\geq 0)$ with $a_0=1$ is a real-valued sequence. Then $\q_k^{[n]}\in \Pi_{k+n}$ and its coefficients are given by
 \begin{align*}
  \q_k^{[n]}[j]
  =\sum_{m=j}^{k+n}{m+1 \choose j} \left(1
  -(m+1-j)\sum_{\ell=0}^{n-1} a_{\ell}\,\ell!\, \stir{m-j}{\ell}
  \right)\tau_{k+n+1}[m+1],
 \end{align*}
 $j=0,\ldots,k+n$.
\end{prop}
\begin{proof}
The degree of $\q_k^{[n]}$ is $k+n$, since by \Cref{lem:D_Delta_coeff} each application of the operators $D$ and $\Delta$ decrease the degree by $1$. Furthermore, from \Cref{lem:D_Delta_coeff} and \Cref{lem:iterated_Delta}, for $j=0,\ldots,k+n$ we obtain
 \begin{align*}
  \q_k^{[n]}[j]=&\sum_{m=j}^{k+n} {m+1 \choose j}\tau_{k+n+1}[m+1]\\
  &-\sum_{\ell=0}^{n-1}a_{\ell}\sum_{s=j}^{k+n}(s+1)\tau_{k+n+1}[s+1]{s \choose j}\ell! \, \stir{s-j}{\ell}\\
  =&\sum_{m=j}^{k+n} \left({r+1 \choose j}
  -(r+1){r \choose j}\sum_{\ell=0}^{n-1}a_{\ell} \,
  \ell! \, \stir{m-j}{\ell}\right)\tau_{k+n+1}[m+1]\\
  =&\sum_{m=j}^{k+n}{m+1 \choose j} \left(1
  -(m+1-j)\sum_{\ell=0}^{n-1}a_{\ell} \,
  \ell! \, \stir{m-j}{\ell}\right)\pi_{k+n+1}[m+1],
 \end{align*}
 which proves the claim.
\end{proof}
The following observation is the key result that makes our construction work (see the proof of \Cref{thm_main}): If in \Cref{prop:q} the elements of the sequence $(a_n:n\geq 0)$ are the Gregory coefficients, then the degree of the polynomial $\sigma^{[n]}_k$ is $k$:
\begin{prop}\label{prop:degree_q}
Let $(\tau_k: k\geq 0)$ be a sequence of polynomials such that $\tau_k \in \Pi_k,\, k\geq 0$. For $n\geq 1, \, k \geq 0$, define
 $\q_k^{[n]}=\Delta \tau_{k+n+1}-\sum_{\ell=0}^{n-1}G_{\ell}
 \Delta^{\ell}D\tau_{k+n+1}$, where $G_n$ are the Gregory coefficients \eqref{def:greg}. Then $\q_k^{[n]}\in \Pi_k$.
\end{prop}
\begin{proof}
Note that $\q_k^{[n]}$ is the same sequence as in \Cref{prop:q} with $a_n=G_n$ and that $G_0=1$. Therefore we know that $\de(\q_k^{[n]})=n+k$.
In order to prove the statement of this Proposition, we have to show that
\begin{equation*}
\q_k^{[n]}[j]=0, \quad j=k+1,\ldots, k+n.
\end{equation*}
This can be deduced from the following observation: For $j\in \NN$ and $m=j,\ldots,k+n$, we have $m-j=0,\ldots, k+n-j$. Now if $j=k+1,\ldots,k+n$, then $m-j\in \{0,\ldots,n-1\}$. In particular, $m-j \leq n-1$. Therefore, in this case
\begin{equation}\label{eq:greg}
\sum_{\ell=0}^{n-1}G_{\ell}\,
  \ell! \, \stir{m-j}{\ell}
  =\sum_{\ell=0}^{m-j}G_{\ell}\,
  \ell! \, \stir{m-j}{\ell}
  =\frac{1}{m-j+1},
\end{equation}
using that $\stir{m-j}{\ell}=0$ for $\ell \geq m-j$ (see \eqref{prop:stir}) and the relation \eqref{relation_stir_greg} between the Stirling numbers of the second kind and the Gregory coefficients. Now from the form of $\q_k^{[n]}[j]$ in \Cref{prop:q}, we see that \eqref{eq:greg} implies the vanishing of $\q_k^{[n]}[j]$ for $j=k+1,\ldots,k+n$.
\end{proof}
%

\section{Statement and proof of the main results}\label{sec:proof}
%
The main results of this paper are formulated and proved in \Cref{prop:case0}, \Cref{thm_main} and \Cref{prop:greg_n}. They show that every Hermite subdivision operator of spectral order $d$ can be factorized as in \Cref{lem:iteration_hermite}, and that this factorization is with respect to the Gregory operators \eqref{def:gregory_op}. Furthermore, we give an explicit characterization of the eigenspaces in \Cref{lem:iteration_hermite}, and an easy-to-check criterion for $C^d$-convergence of an Hermite subdivision scheme of spectral order $d$ (\Cref{cor:greg_d}).
\begin{prop}\label{prop:case0}
Let $S_{\A}$ be an Hermite subdivision operator of spectral order $d$, $d\geq 1$. Denote by $\fp_k, k=0,\ldots,d$, its spectral polynomials (\Cref{def:spectral}).
Then there exists a subdivision operator $S_{\B^{[0]}}$ such that
\begin{align}\label{al:taylor}
    	TS_{\A}=2^{-1}S_{\B^{[0]}}T,  
\end{align}
with $\E_{\B^{[0]}}$ spanned by $[0,1]^T$.
Furthermore, for $k=0,\ldots,d-1$, the polynomials 
\begin{align*}
p_k^{[0]}&:=\Delta \fp_{k+1}-D\fp_{k+1}, \\
q_k^{[0]}&:=D\fp_{k+1},
\end{align*}
satisfy
\begin{equation}\label{al:V_factor}
    S_{\B^{[0]}}\left[  \begin{array}{cc}
                          p_k^{[0]}\\
                          q_k^{[0]}
                         \end{array}
            \right]
     = 2^{-k}
            \left[  \begin{array}{cc}
                          p_k^{[0]}\\
                          q_k^{[0]}
                         \end{array}
            \right]
\end{equation}
and $p_0^{[0]}=0,q_0^{[0]}=1$. If $d>1$, then $p_k^{[0]} \in \Pi_{k-1}$,
$q_k^{[0]}\in \Pi_k$ for $k=1,\ldots, d-1$.
\end{prop}
\begin{proof}
The existence of $S_{\B^{[0]}}$ as well as the form of the eigenspace follow from \cite{merrien12} (which is summarized in our \Cref{thm:Taylor}). We prove the part involving the polynomials $p_k^{[0]},q_k^{[0]}$. By definition
\begin{align*}
\left[ \begin{array}{c}
               p_k^{[0]}\\
               q_k^{[0]}
              \end{array}
	\right]=
T
\left[ \begin{array}{c}
               \fp_{k+1}\\
               D\fp_{k+1}
              \end{array}
	\right], \quad k=0,\ldots,d-1.
\end{align*}
Therefore
\begin{align*}
 S_{\B^{[0]}}\left[  \begin{array}{cc}
                      p_k^{[0]}\\
                      q_k^{[0]}
                     \end{array}
	    \right]
 &= S_{\B^{[0]}} T \left[ \begin{array}{c}
               \fp_{k+1}\\
               D\fp_{k+1}
              \end{array}
	\right]
= 2 T S_{\A}\left[ \begin{array}{c}
               \fp_{k+1}\\
               D\fp_{k+1}
              \end{array}
	\right]
= 2 \cdot 2^{-k-1}T\left[ \begin{array}{c}
               \fp_{k+1}\\
               D\fp_{k+1}
              \end{array}
	\right]	\\
&= 2^{-k} \left[ \begin{array}{c}
               p^{[0]}_k\\
               q^{[0]}_k
              \end{array}
	\right].	
\end{align*}
It is easy to see that $q^{[0]}_0=1$. The degree of $q^{[0]}_k$ is $k$ since $D$ decreases the degree of $\fp_{k+1}$ by $1$. \Cref{cor:Delta_minus_D} implies that $p_0^{[0]}=0$ and that the degree of $p_k^{[0]}$ is $k-1$, $k=1,\ldots, d-1$, if $d>1$.
\end{proof}
\begin{thm} \label{thm_main}
Let $d\geq 2$ and let $S_{\A}$ be an Hermite subdivision operator of spectral order $d$. Denote by $\fp_k, k=0,\ldots,d$, its spectral polynomials (\Cref{def:spectral}). Then we have the following:
\begin{enumerate}
\item
For $n=1,\ldots,d-1$ there exist subdivision operators $S_{\B^{[n]}}$  such that
\begin{align}\label{al:taylor2}
  \Delta_{V^{[n-1]}} \cdots \Delta_{V^{[0]}} T 			
  S_{\A}=2^{-(n+1)}S_{\B^{[n]}}\Delta_{V^{[n-1]}} \cdots 
  \Delta_{V^{[0]}} T,
\end{align} 
where
\begin{equation*}
  V^{[0]}=\left[
	    \begin{array}{cc}
	     0 & 1 \\
	     1 & 0
	    \end{array}
	  \right], \quad
  V^{[n]}=\left[
	    \begin{array}{cc}
	     1 & 0 \\
	     G_n & 1
	    \end{array}
	  \right], \quad n=1,\ldots,d-1,   
\end{equation*}
and $G_{n}$ are the Gregory coefficients of \eqref{def:greg}.
Furthermore, the eigenspaces $\E_{\B^{[n]}}$ are spanned by $[1,G_n]^T$.
\item
For $n=1,\ldots,d-1$ we define polynomials $p_k^{[n]},q_k^{[n]}$ by
 \begin{align*}
 p_k^{[n]} &:=\Delta^{n}D\fp_{k+n+1}, \quad k=0,\ldots,d-1-n,\\ 
 q_k^{[n]} &:=\Delta \fp_{k+n+1}-\sum_{\ell=0}^{n-1}G_{\ell}
 \Delta^{\ell}D\fp_{k+n+1}, \quad k=0,\ldots,d-1-n.
\end{align*}

They satisfy
        \begin{align}\label{al:poly_reprod}
         &S_{\B^{[n]}}\left[  \begin{array}{cc}
                              p_k^{[n]}\\
                              q_k^{[n]}
                             \end{array}
                \right]
         = 2^{-k}
                \left[  \begin{array}{cc}
                              p_k^{[n]}\\
                              q_k^{[n]}
                             \end{array}
                \right], \quad k=0,\ldots, d-1-n 
     \end{align}
and $p_k^{[n]},q_k^{[n]} \in \Pi_k$.
\end{enumerate}
\end{thm}
\begin{Rem}
Note that $p_k^{[n]},q_k^{[n]}$ are exactly the sequences of \Cref{prop:p} and \Cref{prop:q} using $\tau_k=\fp_k$ and $a_n=G_n$.
\end{Rem}
\begin{proof}
We fix $d\geq 2$ and prove this theorem by recursion on $n$, making use of \Cref{prop:case0}.

\textit{Case } $n=1$:
\Cref{prop:case0} implies that there exists $S_{\B^{[0]}}$ such that $TS_{\A}=2^{-1}S_{\B^{[0]}}T$ and $\E_{\B^{[0]}}$ is spanned by $[0,1]^T$. Therefore $V^{[0]}$ is an $\E_{\B^{[0]}}$-generator and by \Cref{thm:vector} there exists $S_{\B^{[1]}}$ such that 
\begin{equation*}
 \Delta_{V^{[0]}} S_{\B^{[0]}}=2^{-1}S_{\B^{[1]}}\Delta_{V^{[0]}}.
\end{equation*}
Therefore
\begin{equation*}
\Delta_{V^{[0]}} T 			
  S_{\A}=2^{-1}\Delta_{V^{[0]}}S_{\B^{[0]}}T
  =2^{-2}S_{\B^{[1]}}\Delta_{V^{[0]}}T,
 \end{equation*}
for $k=0,\ldots,d-2$, which proves \eqref{al:taylor2}.

With $p_{k}^{[0]},q_{k}^{[0]}$ defined in \Cref{prop:case0}, we have
\begin{equation}\label{eq:def_p1_q1}
 p_k^{[1]}=\Delta q_{k+1}^{[0]}, \quad q_k^{[1]}=p_{k+1}^{[0]}, \quad k=0,\ldots,d-2.
\end{equation}
That is,
\begin{align*}
\left[ \begin{array}{c}
               p_k^{[1]}\\
               q_k^{[1]}
              \end{array}
	\right]
=
\left[ \begin{array}{cc}
  0 & \Delta \\
  1 & 0
 \end{array}
 \right]
\left[ \begin{array}{c}
               p_{k+1}^{[0]}\\
               q_{k+1}^{[0]}
              \end{array}
	\right]
=
\Delta_{V^{[0]}}
\left[ \begin{array}{c}
               p_{k+1}^{[0]}\\
               q_{k+1}^{[0]}
              \end{array}
	\right].    
\end{align*}
Again using \Cref{prop:case0}, this implies
\begin{align}\label{al:eigen}
 S_{\B^{[1]}}\left[  \begin{array}{c}
                      p_k^{[1]}\\
                      q_k^{[1]}
                     \end{array}
	    \right]
& = S_{\B^{[1]}}\Delta_{V^{[0]}}  \left[ \begin{array}{c}
               p_{k+1}^{[0]}\\
               q_{k+1}^{[0]}
              \end{array}
	\right]
= 2\Delta_{V^{[0]}} S_{\B^{[0]}}\left[ \begin{array}{c}
               p_{k+1}^{[0]}\\
               q_{k+1}^{[0]}
              \end{array}
	\right]\\ \nonumber
&= 2^{-k}\Delta_{V^{[0]}}\left[ \begin{array}{c}
               p^{[0]}_{k+1}\\
               q^{[0]}_{k+1}
              \end{array}
	\right]
= 2^{-k}\left[  \begin{array}{c}
                      p_k^{[1]}\\
                      q_k^{[1]}
                     \end{array}
	    \right],
 \end{align}
 proving \eqref{al:poly_reprod}.

From \Cref{prop:case0} we know that $q_{k+1}^{[0]} \in \Pi_{k+1}$ and $p_{k+1}^{[0]}\in \Pi_k$ and thus \eqref{eq:def_p1_q1} implies that $p_{k}^{[1]},q_{k}^{[1]} \in \Pi_k$, $k=0,\ldots,d-2$.

In particular $p^{[1]}_0,q^{[1]}_0$ are constants.
From the computation \eqref{al:eigen} we see that $[p^{[1]}_0,q^{[1]}_0]^T$ lies the eigenspace of $S_{\B^{[1]}}$ with respect to the eigenvalue $1$. Using the explicit form of $p^{[1]}_0,q^{[1]}_0$ from\Cref{prop:p} and \Cref{prop:q} we compute
 \begin{align*}
  &p_0^{[1]}=p_0^{[1]}[0]=1,\\
  &q_0^{[1]}=q_0^{[1]}[0]=1/2=G_1.
 \end{align*}
Hence the eigenspace $\E_{\B^{[1]}}\neq \{0\}$ and by \Cref{thm:vector} it has dimension $1$. Therefore it is spanned by $[p^{[1]}_0,q^{[1]}_0]^T=[1,G_1]^T$.
This concludes the proof for $n=1$.

If $d=2$ nothing else needs to be shown, since $n=1$ only. We now prove that if $d>2$ the case $n-1$ implies the case $n$, for $n=2,\ldots,d-1$.

We assume that the statements of the theorem are satisfied for $n-1$ and prove it for $n$.

By assumption there exists a subdivision operator $S_{\B^{[n-1]}}$ such that 
\begin{align*}
  \Delta_{V^{[n-2]}} \cdots \Delta_{V^{[0]}} T 			
  S_{\A}=2^{-n}S_{\B^{[n-1]}}\Delta_{V^{[n-2]}} \cdots 
  \Delta_{V^{[0]}} T
\end{align*}
and the eigenspace $\E_{\B^{[n-1]}}$ is spanned by $[1,G_{n-1}]^T$. Therefore $V^{[n-1]}$ is an $\E_{\B^{[n-1]}}$-generator and by \Cref{thm:vector} there exists a subdivision operator $S_{\B^{[n]}}$ such that
\begin{equation*}
 \Delta_{V^{[n-1]}} S_{\B^{[n-1]}}=2^{-1}S_{\B^{[n]}}\Delta_{V^{[n-1]}}.
\end{equation*}
This implies
\begin{align*}
  \Delta_{V^{[n-1]}} \cdots \Delta_{V^{[0]}} T 			
  S_{\A}
  &=2^{-n}\Delta_{V^{[n-1]}}S_{\B^{[n-1]}}\Delta_{V^{[n-2]}} \cdots 
  \Delta_{V^{[0]}} T\\
  &=2^{-(n+1)}S_{\B^{[n]}}\Delta_{V^{[n-1]}}\Delta_{V^{[n-2]}} \cdots 
  \Delta_{V^{[0]}} T,
\end{align*}
which proves \eqref{al:taylor2}.

Since $p_k^{[n]},q_k^{[n]}$ are the sequences of \Cref{prop:p} and \Cref{prop:q} with $\tau_k=\fp_k$ and $a_n=G_n$, we know from \Cref{lem:recursion_pq} that
\begin{align*}
\left[ \begin{array}{c}
               p_k^{[n]}\\
               q_k^{[n]}
              \end{array}
	\right]=
\Delta_{V^{[n-1]}}
\left[ \begin{array}{c}
               p_{k+1}^{[n-1]}\\
               q_{k+1}^{[n-1]}
              \end{array}
	\right], \quad k=0,\ldots, d-1-n.
\end{align*}
Therefore
\begin{align}\label{al:n}
 S_{\B^{[n]}}\left[  \begin{array}{c}
                      p_k^{[n]}\\
                      q_k^{[n]}
                     \end{array}
	    \right]
& = S_{\B^{[n]}}\Delta_{V^{[n-1]}}  \left[ \begin{array}{c}
               p_{k+1}^{[n-1]}\\
               q_{k+1}^{[n-1]}
              \end{array}
	\right]
= 2\Delta_{V^{[n-1]}} S_{\B^{[n-1]}}\left[ \begin{array}{c}
               p_{k+1}^{[n-1]}\\
               q_{k+1}^{[n-1]}
              \end{array}
	\right]\\ \nonumber
&= 2^{-k}\Delta_{V^{[n-1]}}\left[ \begin{array}{c}
               p^{[n-1]}_{k+1}\\
               q^{[n-1]}_{k+1}
              \end{array}
	\right]
= 2^{-k}\left[  \begin{array}{c}
                      p_k^{[n]}\\
                      q_k^{[n]}
                     \end{array}
	    \right],
 \end{align}
 which proves \eqref{al:poly_reprod}.
 
Since $\fp_k\in \Pi_k$, from \Cref{prop:p} and \Cref{prop:degree_q} we can conclude that $p^{[n]}_k,q^{[n]}_k\in \Pi_k$. 

In particular, $p^{[n]}_0,q^{[n]}_0$ are constants and from the computation \eqref{al:n} we see that $[p^{[n]}_0,q^{[n]}_0]^T$ lies in the eigenspace of $S_{\B^{[n]}}$ with respect to the eigenvalue $1$.
Using the explicit formula of \Cref{prop:p} we get
\begin{align*}
 p^{[n]}_0
 &=p^{[n]}_0[0]
=n!\, \sum_{m=0}^{n}(m+1)\fp_{n+1}[m+1]{m \choose 0}\stir{m}{n}\\
&=n!\, (n+1)\fp_{n+1}[n+1]\stir{n}{n}=\frac{(n+1)!}{(n+1)!}\\
&=1,
\end{align*}
where we use the properties of the Stirling numbers of the second kind \eqref{prop:stir} and the fact that $\fp_{\ell}[\ell]=1 / \ell !$ from \Cref{def:spectral}.

Continuing with the explicit formula for $q^{[n]}_0$ from \Cref{prop:q} we also get
\begin{align*}
 q^{[n]}_0=&\, q^{[n]}_0[0]\\
 =& \sum_{m=0}^{n}{m+1 \choose 0} \left(1
  -(m+1)\sum_{\ell=0}^{n-1} G_{\ell}\ell!\, \stir{m}{\ell}
  \right)\fp_{n+1}[m+1]\\
 =& \sum_{m=0}^{n} \left(1
  -(m+1)\sum_{\ell=0}^{\min\{m,n-1\}} G_{\ell}\ell!\, \stir{m}{\ell}
  \right)\fp_{n+1}[m+1]\\  
=& \sum_{m=0}^{n-1}\left(1
  -(m+1)\sum_{\ell=0}^{m} G_{\ell}\ell!\, \stir{m}{\ell}
  \right)\fp_{n+1}[m+1]\\
  &+
  \left(1
  -(n+1)\sum_{\ell=0}^{n-1} G_{\ell}\ell!\, \stir{n}{\ell}
  \right)\fp_{n+1}[n+1]\\  
=& \sum_{m=0}^{n-1}\left(1
  -(m+1)\frac{1}{m+1}\right)\fp_{n+1}[m+1]\\
  &+
  \left(1
  -(n+1)\left(\frac{1}{n+1}-G_n n! \stir{n}{n}\right)
  \right)\fp_{n+1}[n+1]\\  
=& \,(n+1)\,G_n\, n! \, \fp_{n+1}[n+1]\\
=& \,G_n,
\end{align*}
where again we use the properties of the Stirling numbers of the second kind \eqref{prop:stir}, the relation \eqref{relation_stir_greg}, and the fact that $\fp_{\ell}[\ell]=1 / \ell !$ from \Cref{def:spectral}.

The eigenspace $\E_{\B^{[n]}}\neq \{0\}$ and thus, by \Cref{thm:vector}, it has dimension $1$. It is therefore spanned by $[p_0^{[n]},q_0^{[n]}]^T=[1,G_n]^T$. This concludes the proof.
\end{proof}
\begin{cor}
With notation as in \Cref{thm_main}, the operator used for factorizing is the $n$-th Gregory operator \eqref{def:gregory_op}, that is
\begin{equation*}
\Delta_{V^{[n-1]}}\cdots \Delta_{V^{[0]}}T
=\G^{[n]}=\left[\begin{array}{cc}
		0 & \Delta^n \\
        \Delta & -\sum_{\ell=0}^{n-1}G_{\ell}\Delta^{\ell}
		\end{array}\right], \quad n=1,\ldots,d-1.
\end{equation*}
\end{cor}
From \Cref{prop:case0} and \Cref{thm_main} we get one additional factorization:

\begin{prop}\label{prop:greg_n}
Let $d\geq 1$ and let $S_{\A}$ be an Hermite subdivision operator of spectral order $d$. Then for $n=1,\ldots, d$, the operator $S_{\A}$ factorizes with respect to the $n$-th Gregory operator $\G^{[n]}$ \eqref{def:gregory_op}, that is, there exist subdivision operators $S_{\B^{[n]}}$ such that
\begin{equation*}
\G^{[n]}S_{\A}=2^{-(n+1)}S_{\B^{[n]}}\G^{[n]}, \quad n=1,\ldots, d.
\end{equation*}
\end{prop}
\begin{proof} 
For $d=1$, from \Cref{prop:case0}, we get a factorization
\begin{equation*}
TS_{\A}=2^{-1}S_{\B^{[0]}}T,
\end{equation*}
and we know that $\E_{\B^{[0]}}$ is spanned by $[0,1]^T$. Therefore $V^{[0]}$ defined in \Cref{thm_main} is an $\E_{\B^{[0]}}$-generator and by \Cref{thm:vector} there exists $S_{\B^{[1]}}$ such that
\begin{equation*}
\Delta_{V^{[0]}}TS_{\A}=2^{-2}S_{\B^{[1]}}\Delta_{V^{[0]}}T,
\end{equation*}
and $\Delta_{V^{[0]}}T=\G^{[1]}$. This concludes the case $d=1$.

Now if $d\geq 2$, from \Cref{thm_main} we obtain $S_{\B^{[n]}}$ such that
\begin{equation*}
\G^{[n]}S_{\A}=2^{-(n+1)}S_{\B^{[n]}}\G^{[n]}, \quad n=1,\ldots,d-1.
\end{equation*}
We know that $\E_{\B^{[d-1]}}$ is spanned by $[1,G_{d-1}]^T$. Therefore, by \Cref{thm:vector}, we can factorize with respect to 
\begin{equation*}
V^{[d-1]}= \left[ \begin{array}{cc}
			1 & 0 \\
            G_{d-1} & 1
			\end{array}
          \right]
\end{equation*}
and obtain that
there exists a subdivision operator $S_{\B^{[d]}}$ such that
\begin{equation*}
\G^{[d]}S_{\A}=2^{-(d+1)}S_{\B^{[d]}}\G^{[d]}.
\end{equation*}
This concludes the proof for $d\geq 2$.
\end{proof}
\begin{Rem}
Note that with $S_{\C^{[n]}}=2^{-1}S_{\B^{[n]}}$, \Cref{prop:greg_n} is exactly the main theorem stated in the introduction (\Cref{intro:main}).
\end{Rem}

The factorization of \Cref{prop:greg_n} together with the Taylor factorization of \Cref{prop:case0} satisfies the assumptions of \Cref{lem:iteration_hermite}.
Therefore we get a criterion to check the $C^d$-convergence, $d\geq 1$, of an Hermite subdivision scheme of spectral order $d$: If $2^{-1}S_{\B^{[d]}}$ is contractive, then the Hermite scheme associated with $S_{\A}$ is $C^d$-convergent.
By considering $S_{\B}=2^{-1}S_{\B^{[d]}}$, we thus obtain \Cref{intro:cor_main} from the introduction:
\begin{cor}\label{cor:greg_d}
Let $d\geq1$ and let $S_{\A}$ be an Hermite subdivision operator of spectral order $d$. Then there exists $S_{\B}$ such that
\begin{equation*}
\G^{[d]}S_{\A}=2^{-d}S_{\B}\G^{[d]},
\end{equation*}
where $\G^{[d]}$ is the $d$-th Gregory operator \eqref{def:gregory_op}.
If $S_{\B}$ is contractive, then the Hermite subdivision scheme associated with $S_{\A}$ is $C^d$-convergent.
\end{cor}
\begin{Rem}
Note that since the spectral condition of order $d$ implies the spectral condition of order $\ell$, for every $\ell \leq d$, \Cref{cor:greg_d} can be used to prove any regularity $\ell \leq d$ of the Hermite scheme. This is useful for schemes which have lower regularity than polynomial reproduction order, see e.g.\ some of the examples in \cite{jeong17}.
\end{Rem}
\section{Examples}\label{sec:ex}
In this section we provide an algorithm for computing the $n$-th Gregory factorization using symbols and apply it to an example of \cite{jeong17}. We also show that this example is an incident of an Hermite scheme which satisfies the spectral condition but does not reproduce polynomials, proving that these concepts are not equivalent.
\begin{alg}\label{alg:general_ex}
We show how the $n$-th Gregory factorization \eqref{nth_greg} can be computed using \emph{symbols}. The \emph{symbol} of a sequence $\cb\in \ell(\ZZ)_0^{2}$ is the Laurent polynomial
\begin{equation*}
\cb^{\ast}(z)=\sum_{j\in\ZZ}\cb_j z^j, \quad z\in \CC \,\backslash \,\{0\}.
\end{equation*}
Similarly, we can define $\A^{\ast}(z)$ for $\A \in \ell(\ZZ)^{2\times 2}_{0}$. It is well-known \cite{dyn02,merrien12} that a factorization of the form \eqref{nth_greg} relates to the following equation in symbols:
\begin{equation*}
{\G^{[n]}}^{\ast}(z)\A^{\ast}(z)=2^{-n}{\B^{[n]}}^{\ast}(z){\G^{[n]}}^{\ast}(z^2).
\end{equation*}
With $\A^{\ast}(z)=\left[\ab^{\ast}_{jk}(z) \right]_{j,k=1}^2$ and $g^{\ast}_n(z)=-\sum_{\ell=0}^{n-1}G_{\ell}(z^{-1}-1)^{\ell}$ we obtain ${\B^{[n]}}^{\ast}(z)=\left[{\bb^{[n]}_{jk}}^{\ast}(z) \right]_{j,k=1}^2$:
\begin{align*}
&{\bb^{[n]}_{11}}^{\ast}(z)=2^{n}\frac{(z^{-2}-1)\ab_{22}^{\ast}(z)-g^{\ast}_n(z^2)\ab_{21}^{\ast}(z)}{(z^{-1}-1)(z^{-1}+1)^{n+1}},\\[0.2cm]
&{\bb^{[n]}_{12}}^{\ast}(z)=2^{n}\frac{(z^{-1}-1)^{n-1}\ab_{21}^{\ast}(z)}{z^{-1}+1},\\[0.2cm]
&{\bb^{[n]}_{21}}^{\ast}(z)=2^{n}\\
&\frac{(z^{-2}-1)((z^{-1}-1)\ab_{12}^{\ast}(z)+g^{\ast}_n(z)\ab^{\ast}_{22}(z))-g^{\ast}_n(z^2)((z^{-1}-1)\ab_{11}^{\ast}(z)+g^{\ast}_n(z)\ab_{21}^{\ast}(z))}{(z^{-2}-1)^{n+1}},\\[0.2cm]
&{\bb^{[n]}_{22}}^{\ast}(z)=2^{n}\frac{(z^{-1}-1)\ab_{11}^{\ast}(z)+g^{\ast}_n(z)\ab_{21}^{\ast}(z)}{(z^{-2}-1)},
\end{align*}
which can be computed, for example, with Mathematica.
\end{alg}

\begin{ex}\label{ex:H1}
We consider the primal Hermite subdivision scheme $H_1$ proposed in \cite{jeong17}. Its mask is supported in $[-2,2]\cap \ZZ$ with nonzero elements given by
\begin{align*}
\left[
\begin{array}{rr} \theta & -\frac{\theta}{2}\\ -\frac{3\omega}{2} & \frac{\omega}{2}\end{array}\right], \,
\left[
\begin{array}{rr} \frac12 & -\frac18\\[0.1cm] \frac34& -\frac{1}{8}\end{array}\right], \,
\left[
\begin{array}{cc} 1-2\theta & 0\\ 0 & \frac{1+4\omega}{2}\end{array}\right], \,
\left[
\begin{array}{rr} \frac12 & \frac18\\[0.1cm] -\frac34& -\frac{1}{8}\end{array}\right], \,
\left[
\begin{array}{rr} \theta & \frac{\theta}{2}\\[0.05cm] \frac{3\omega}{2}& \frac{\omega}{2}\end{array}\right],
\end{align*}
with parameters $\theta, \omega \in \RR$.

In \cite{jeong17} it is proved that $H_1$ reproduces polynomials up to degree $3$ and thus it satisfies the spectral condition up to order $3$ with spectral polynomials $1,x,\tfrac{1}{2!}x^2,\tfrac{1}{3!}x^3$, see \cite{conti14}. By \Cref{prop:greg_n}, we have Gregory factorizations for $n=1,2,3$.

It is easy to see that the scheme $H_1$ does not satisfy the spectral condition of order $4$ with spectral polynomial $\tfrac{1}{4!}\,x^4$ for all parameters $\theta,\omega$. This implies that it does not reproduce polynomials of degree $4$ for all parameters $\theta, \omega$, see \cite{conti14}. This can also be proved using the methods of \cite{conti18}.
However, with $\theta=1/32$ it satisfies the spectral condition of order $4$ with $4$-th spectral polynomial given by $\fp_{4}(x)=\tfrac{1}{4!}\,x^4+\tfrac{1}{360}$. Therefore, $H_1$ with $\theta=1/32$ provides an example of an Hermite scheme which does not reproduce polynomials of degree $4$, but satisfies the spectral condition of order $4$. 
To the best of our knowledge, this is the first time it is observed that the spectral condition is \emph{not} equivalent to the reproduction of polynomials.

This explains why in \cite{jeong17} a factorization of $H_1$ 
up to order $n=4$ is possible, even though the mask only reproduces polynomials up to degree $3$. Of course we also have a $4$-th Gregory factorization for $\theta=1/32$, which we now provide using \Cref{alg:general_ex}.

For $\theta=1/32$, the mask $\B^{[4]}$ of the $4$-th Gregory factorization \eqref{nth_greg} is supported in $[-4,2]\cap \ZZ$:
\begin{align*}
\left[
\begin{array}{cc}
0 & -24\,\omega\\
0 & 0
\end{array}
\right], \,
\left[
\begin{array}{cc}
0 & 96\,\omega+12\\[0.1cm]
0 & \omega
\end{array}
\right], \,
\left[
\begin{array}{cc}
-\omega & -168\,\omega-48\\[0.1cm]
0 & -5\,\omega-\frac{1}{2}
\end{array}
\right],
\end{align*}
\vspace{-0.4cm}
\begin{align*}
\left[
\begin{array}{cc}
4\,\omega +\frac{1}{2} & 192\,\omega+72\\[0.1cm]
\frac{\omega}{24} & 20\,\omega+3
\end{array}
\right],\,
\left[
\begin{array}{cc}
-6\,\omega-2 & -168\,\omega-48 \\[0.1cm]
-\frac{5\,\omega}{24}-\frac{1}{48} & 4\,\omega-2
\end{array}
\right],
\end{align*}
\vspace{-0.4cm}
\begin{align*}
\left[
\begin{array}{cc}
4\,\omega +\frac{5}{2} & 96\,\omega+12\\[0.1cm]
\frac{19\,\omega}{24}+\frac{1}{8} & 19\,\omega+3
\end{array}
\right], \quad
\left[
\begin{array}{cc}
-\omega & -24\,\omega\\[0.1cm]
\frac{3\,\omega}{8}-\frac{1}{16} & 9\,\omega +\frac{1}{2}
\end{array}
\right].
\end{align*}
We would like to stress that $\theta=1/32$ is the only value for which we obtain a $4$-th Gregory factorization of the scheme $H_1$ (and thus the only value for which the spectral condition of order $4$ is satisfied).
With the results of \Cref{lem:iteration_hermite} and our Gregory factorzation we can now analyze the smoothness of $H_1$. Numerical computations show that $\Vert(\frac{1}{2}S_{\B^{[4]}})^6\Vert_\infty<1$ for $\omega \in [-0.10210,-0.09582]$. Thus $H_1$ is $C^4$ for this range of $\omega$ which confirms the result of \cite{jeong17}. The advantage of our factorization however is that we only need 6 iterations to prove the contractivity of $S_{\B^{[4]}}$ whereas 24 iterations are needed in \cite{jeong17}. Therefore, we can enlarge the domain for $\omega$ and still obtain a smoothness result. Computations show that the Hermite scheme $H_1$ is $C^4$ for $\omega \in [-0.12,-0.088]$ since $\Vert(\frac{1}{2}S_{\B^{[4]}})^{10}\Vert_\infty<1$ for these values of $\omega$.
\end{ex}

\section{Conclusion}\label{sec:conclusion}
In this paper we provide a novel factorization framework for Hermite subdivision operators based on Stirling numbers and Gregory coefficients. We further derive \Cref{alg:general_ex}, which allows to easily compute the $n$-th Gregory factorization using symbols. The usefulness of the Gregory factorization is evident from the reduction of computational cost for proving $C^d$-convergence of an Hermite subdivision scheme: Only one factorization needs to be computed, independently of $d$ (\Cref{cor:greg_d}).
Certainly, the $d$-th Gregory factorization is not the only possible factorization for Hermite schemes of spectral order $d$, but the only one which is explicitly computed for general $d$.

Furthermore, in \Cref{ex:H1}, we provide an instance of an Hermite scheme which satisfies the spectral condition of order $d=4$, but does not reproduce polynomials of degree $4$, showing that the spectral condition is not equivalent to the reproduction of polynomials.

\section*{Acknowledgments}
We thank B.\ Jeong and J.\ Yoon for sharing some of the masks from \cite{jeong17}, which we used for validating our method.

\noindent
S.H.\ acknowledges the support of the Austrian Science Fund (FWF): W1230.

\noindent
C.C.\ acknowledges the support of GNCS-INdAM, Italy.

\bibliographystyle{abbrvnat}

\end{document}